\documentclass[a4paper,12pt]{article}

\usepackage[pdftex]{graphicx}
\usepackage[section] {placeins}
\usepackage{enumerate}
\usepackage[normalem]{ulem}
\usepackage{amsmath}
\usepackage{amsthm}
\usepackage{latexsym}
\usepackage{amsfonts}
\usepackage{amssymb}
\usepackage{mathtools}
\usepackage{longtable}

\usepackage{tikz-cd}
\usetikzlibrary{graphs}
\usetikzlibrary{arrows.meta}
\usetikzlibrary{shapes.geometric}

\usepackage{authblk}

\usepackage{latexsym,amssymb,amsfonts, amsthm, amsmath}
\usepackage[english]{babel}
\usepackage{bm}
\usepackage{mathrsfs}
\usepackage{algpseudocode}
\usepackage{algorithm}

\algnewcommand{\IIf}[1]{\State\algorithmicif\ #1\ \algorithmicthen}
\algnewcommand{\EndIIf}{\unskip\ \algorithmicend\ \algorithmicif}

\newtheorem{thm}{Theorem}[section]
\newtheorem{lem}[thm]{Lemma}
\newtheorem{cor}[thm]{Corollary}
\newtheorem{prop}[thm]{Proposition}
\newtheorem{conj}[thm]{Conjecture}

\newtheorem{rem}[thm]{Remark}

\theoremstyle{definition} 
 
\newtheorem{exa}[thm]{Example}

\newcommand{\Z}{\mathbb{Z}}

\def\Sym{\mathrm{Sym}}

\usepackage{fullpage}

\begin{document}

\title{On Sequences in Cyclic Groups with \\ Distinct Partial Sums}

\author[1]{Simone Costa} 
\author[2]{Stefano Della Fiore} 
\author[3]{M.~A.~Ollis\footnote{Corresponding author:  \texttt{matt\_ollis@emerson.edu} }}
\author[4]{Sarah Z.~Rovner-Frydman}

\affil[1]{DICATAM, Sez.~Matematica, Universit\`a degli Studi di Brescia, Via Branze~43, I~25123 Brescia, Italy}
\affil[2]{DII, Universit\`a degli Studi di Brescia, Via Branze~38, I~25123 Brescia, Italy}
\affil[3]{Marlboro Institute for Liberal Arts and Interdisciplinary Studies, Emerson College, 120~Boylston Street, Boston, MA 02116, USA} 
\affil[4]{Marlboro College, P.~O.~Box~A, Marlboro, VT 05344, USA}

\date{}

\maketitle

\begin{abstract}
A subset of an abelian group is {\em sequenceable} if there is an ordering~$(x_1, \ldots, x_k)$ of its elements such that the partial sums~$(y_0, y_1, \ldots, y_k)$, given by $y_0 = 0$ and $y_i = \sum_{j=1}^i x_i$ for $1 \leq i \leq k$, are distinct, with the possible exception that we may have~$y_k = y_0 = 0$.   We demonstrate the sequenceability of subsets of size~$k$ of $\Z_n \setminus \{ 0 \}$ when~$n = mt$ in many cases, including when~$m$ is either prime or has all prime factors larger than~$k! /2$ for~$k \leq 11$ and $t \leq 5$ and for~$k=12$ and~$t \leq 4$.  We obtain similar, but partial, results for~$13 \leq k \leq 15$.   This represents progress on a variety of questions and conjectures in the literature concerning the sequenceability of subsets of abelian groups, which we combine and summarize into the conjecture that if a subset of an abelian group does not  contain~0 then it is sequenceable.
\end{abstract}

\section{Introduction}\label{sec:intro}

Given a subset of an abelian group, is it possible to order the elements of the subset in such a way that the partial sums of the ordering are distinct?  This type of problem goes back at least fifty years and there are several conjectures on the topic, described below.  The successful resolution, or partial resolution,  of these conjectures, has implications in the study of graph decompositions and embeddings and in the construction of Heffter arrays and other combinatorial designs.

We introduce some definitions and notation to make the question precise.  Let~$G$ be an abelian group and let $S$ be a subset of $G \setminus \{ 0 \}$ of size~$k$.  

Let ${\bm x} = (x_1, x_2, \ldots, x_k)$ be an ordering of the elements of~$S$ and define its partial sums ${\bm y} = (y_0, y_1, \ldots, y_k)$ by $y_0=0$ and $y_i = x_1 + \cdots + x_i$ for $i>0$.   Denote the sum of the elements of~$S$ by $\Sigma S$.  As $G$ is abelian, for any ordering of the elements of~$S$ the final partial sum $y_k$  is equal to $\Sigma S$.

If the elements of ${\bm y}$ are distinct, then ${\bm x}$ is a {\em sequencing} of~$S$.  If the elements of ${\bm y}$ are distinct with the exception that $y_0 = 0 = y_k$, then ${\bm x}$ is a {\em rotational sequencing} or {\em R-sequencing} of~$S$.  We sometimes refer to a sequencing as a {\em linear sequencing} to emphasise the distinction from rotational sequencings.  As $G$ is abelian, a subset~$S$ cannot have both a linear and a rotational sequencing.  If~$S$ has one or the other, call it {\em sequenceable}.  If every subset~$S \subseteq G \setminus \{ 0 \}$ is sequenceable then $G$ is {\em strongly sequenceable}.  

This nomenclature is consistent with the definition of sequencing and R-sequencing introduced by Gordon in~1961 and Friedlander, Gordon and Miller in~1978 respectively for the case $S = G \setminus \{ 0 \}$~\cite{FGM78,Gordon61}.  Replacing ``R-" with the more descriptive ``rotational"  was suggested by Ahmed, Azimli, Anderson and Preece in~2011~\cite{AAAP11}.  The term ``strongly sequenceable" was first used in the literature by Alspach and Liversidge in~2020, where they say that Alspach and Kalinowski have posed the problem of determining which groups are strongly sequenceable and also make a conjecture that would imply that every finite abelian group is strongly sequenceable.

We now come to the main conjecture.

\begin{conj}\label{conj:main}
Every abelian group is strongly sequenceable.
\end{conj}

Conjecture~\ref{conj:main} is the amalgamation of several questions and conjectures.  In~1971, Graham asked whether every subset~$S$ of~$\Z_n$, the additively written cyclic group of order~$n$, with $0 \not\in S$ and $\Sigma S = 0$,  has a rotational sequencing when~$n$ is prime~\cite{Graham71}.  Independently of this, in~2016 Archdeacon, Dinitz, Mattern and Stinson conjectured that any subset~$S$ of $\Z_n \setminus \{ 0 \}$ with $\Sigma S =0$ has a rotational sequencing~\cite{ADMS16}.  In~2005, Bode and Harborth published the first results on {\em Alspach's conjecture} that every subset~$S$ of $\Z_n \setminus \{ 0 \}$ with $\Sigma S \neq 0$ has a linear sequencing~\cite{BH05}.  In~2018, Costa, Morini, Pasotti and Pellegrini suggested that these conjectures may be generalized from $\Z_n$ to arbitrary (including infinite) abelian groups~\cite{CMPP18}.

Costa, Morini, Pasotti and Pellegrini also put forward a weaker version of Conjecture~\ref{conj:main} that is sufficient for some applications to Heffter arrays:

\begin{conj}\label{conj:cmpp}
Let~$G$ be an abelian group and let~$S$ be a finite subset of~$G \setminus \{ 0 \}$ such that $\Sigma S = 0$ and $|S \cap \{ x,-x \}| \leq 1$ for any~$x \in G$.  Then~$S$ has a rotational sequencing.
\end{conj}

As suggested by the discussion of nomenclature above, the case $S = G \setminus \{0 \}$ was considered earlier (mostly) than these conjectures.  Gordon posed and solved the question in this instance for linear sequenceability in~1961.   Friedlander, Gordon and Miller conjectured the rotational sequenceability version in~1978~\cite{FGM78} and this was recently resolved by Alspach, Kreher and Pastine~\cite{AKP17}.

Theorem~\ref{th:known} summarizes the main known results concerning Conjecture~\ref{conj:main}. 
\begin{thm}\label{th:known}
Let $G$ be an abelian group of order~$n$ and~$S \subseteq G \setminus \{ 0 \}$ with $|S| =k$.
Then~$S$ is sequenceable in the following cases:
\begin{enumerate}
\item $k \leq 9$~\cite{AL20},
\item $k=10$ when~$n$ is prime~\cite{HOS19}, 
\item $k = n-3$ when  $n$ is prime and  $\Sigma S \neq 0$~\cite{HOS19},   
\item $k = n-2$ when  $G$ is cyclic and  $\Sigma S \neq 0$~\cite{BH05}, 
\item $k = n-1$~\cite{AKP17,Gordon61},
\item $n\leq 21$ and $n\leq 23$ when $\sum S=0$~\cite{CMPP18},
\item  $n\leq 25$ when G is cyclic and $\sum S=0$~\cite{ADMS16}.
\end{enumerate}
Furthermore, if $G$ is a torsion free abelian group, then any subset $S$ of $G \setminus \{ 0 \}$ whose size is at most $11$ is sequenceable~\cite{CP20}.
\end{thm}

In addition, Theorem~\ref{th:main}.1 with~$k=11$ and $t=1$ of the present paper was used in~\cite{CP20} as part of a proof that some instances of the conjecture hold when $k=12$ and $\Sigma S = 0$.

Regarding Theorem~\ref{th:known}.3, we note that the case $\Sigma S = 0$ was not considered in~\cite{HOS19} but, as we shall see in the next section, the calculations there work in this case too, without modification.

The main purpose of this paper is to prove more instances of the conjecture.  In the next section we adapt the polynomial method used in~\cite{HOS19} so that it may be used to prove instances of the conjecture in cyclic groups of order~$pt$, where~$p$ is prime and~$t$ is small.  This requires computing coefficients of monomials in various polynomials, the results of which are summarized in Section~\ref{sec:results}.   As well as composite orders, we are able to push~$k$  higher in the prime case.  In sum, we prove:

\begin{thm}\label{th:main}
Let $n = pt$ with~$p$ prime.  Then subsets~$S$ of size~$k$ of~$\Z_n \setminus \{ 0 \}$ are sequenceable in the following cases:
\begin{enumerate}
\item  $k \leq 11$ and $t \leq 5$,
\item  $k=12$ and $t \leq 4$,
\item  $k= 13$ and $t \in \{2,3\}$, provided~$S$ contains at least one element not in the subgroup of order~$p$,
\item  $k=14$ and $t = 2$,   provided~$S$ contains at least one element not in the subgroup of order~$p$,
\item $k=15$ and $t=2$, provided~$S$ does not contain exactly~$0$,~$1$,~$2$ or~$15$ elements of the subgroup of order~$p$.
\end{enumerate}

\end{thm}

In Section~\ref{sec:infinity} we extend this result to any group of the form $G\times \Z_t$ where $G$ is a torsion-free abelian group and the pairs $(t,k)$ are the ones of Theorem \ref{th:main}. As a consequence, we provide an asymptotic result in $\Z_{mt}$.  More precisely, this allows us to prove:

\begin{thm}\label{th:infty}
Let $n = mt$ where all the prime factors of ~$m$ are bigger than $k!/2$. Then subsets~$S$ of size~$k$ of~$\Z_n \setminus \{0\}$ are sequenceable in the following cases:
\begin{enumerate}
\item  $k \leq 11$ and $t \leq 5$,
\item  $k=12$ and $t \leq 4$,
\item  $k= 13$ and $t \in \{2,3\}$, provided~$S$ contains at least one element not in the subgroup of order~$m$,
\item  $k=14$ and $t = 2$,   provided~$S$ contains at least one element not in the subgroup of order~$m$,
\item $k=15$ and $t=2$, provided~$S$ does not contain exactly~$0$,~$1$,~$2$ or~$15$ elements of the subgroup of order~$m$.
\end{enumerate}
\end{thm}

\section{Applying the Polynomial Method}\label{sec:poly}

The method relies on the Non-Vanishing Corollary to the Combinatorial Nullstellensatz, see~\cite{Alon99,Michalek10}.

\begin{thm}\label{th:pm}{\rm (Non-Vanishing Corollary)}
Let~$F$ be a finite field, and let $ f(x_1, x_2, \ldots, x_k)$ be a polynomial in~$F[x_1, x_2, \ldots, x_k]$.  Suppose the degree~$deg(f)$ of~$f$ is $\sum_{i=1}^k \gamma_i$, where each~$\gamma_i$ is a nonnegative integer, and suppose the coefficient of~$\prod_{i=1}^k x_i^{\gamma_i}$ in~$f$ is nonzero.  If~$C_1, C_2, \ldots, C_k$ are subsets of~$F$ with~$|C_i| > \gamma_i$, then there are $c_1 \in C_1, \ldots, c_k \in C_k$ such that~$f(c_1, c_2, \ldots, c_k) \neq 0$.
\end{thm}

In the notation of the Non-Vanishing Corollary, call the monomial $x_1^{|C_1| - 1} \cdots x_k^{|C_k| - 1}$ the {\em bounding monomial}.  The corollary can be rephrased as requiring the polynomial to include a monomial of maximum degree that divides the bounding monomial (where by ``include" we mean that it has a nonzero coefficient).

To use the Non-Vanishing Corollary we require a polynomial for which the nonzeros correspond to successful solutions to the case of the problem under consideration.

We work in the group $\Z_p \times \Z_t$, where~$p$ is coprime with $t$ (and hence $\Z_p \times \Z_t \cong \Z_{pt}$).  The group~$\Z_p$ is a field, and this plays the role of the field~$F$ in the Non-Vanishing Corollary.

Let $\pi_2: \Z_p \times \Z_t \rightarrow \Z_t$ be the projection map that picks out the second coordinate of an element and for a subset $S \subseteq (\Z_{p} \times \Z_t) \setminus \{ (0,0) \}$ let $\pi_2(S)$ be the multiset $\{ \pi_2(s ) : s  \in S \}$.  Define the {\em type} of~$S$ to be the sequence $\bm{\lambda} = (\lambda_0, \ldots \lambda_{t-1} )$, where $\lambda_i$ is the number of times that $i$ appears in $\pi_2(S)$.  

Let $T$ be a multiset of elements from $\Z_t \setminus \{ 0 \}$ with size~$k$.   Let $\bm{a} = (a_1, \ldots a_k)$ be an arrangement of the elements of~$T$ with partial sums $\bm{b} = (b_0, b_1, \ldots, b_k)$.  If, for each~$i$, the element  $i$ appears at most $r$ times in~${\bf b}$ then ${\bf a}$ is a {\em quotient sequencing} of~$T$ with respect to~$r$.  In our setting, given some~$S \subseteq (\Z_p \times \Z_t) \setminus \{ (0,0) \}$ for which we wish to find a sequencing, we shall be interested in quotient sequencings of $\pi_2(S)$ with respect to~$p$.

When~$S = G \setminus \{ 0 \}$, quotient sequencings have been a useful tool in the construction of sequencings from the earliest papers on the subject; see~\cite{OllisSurvey} for a summary of the history.

Given $S \subseteq (\Z_{p} \times \Z_t) \setminus \{ (0,0) \}$ we first construct a quotient sequencing of $\pi_2(S)$ with respect to~$p$.  We then use the polynomial method to show that there is a sequencing for $S$ that projects elementwise onto that quotient sequencing.  Given such a quotient sequencing $\bm{a} = (a_1, \ldots a_k)$  with partial sums $\bm{b} = (b_0, b_1, \ldots, b_k)$, let 
$$\bm{x_a} = \left(  (x_1,a_1), (x_2,a_2), \ldots, (x_k,a_k) \right)$$
be a putative arrangement of the elements of~$S$ with partial sums
$$\bm{y_a} = \left( (y_0, b_0), (y_1,b_1), \ldots, (y_k,a_k) \right).$$ 
Define a polynomial in variables $x_1, x_2, \ldots, x_k$ by
$$p_{\bm{a}} = \prod_{\substack{1 \leq i < j \leq k \\ a_i = a_j }}  (x_j - x_i)  
                      \prod_{\substack{0 \leq i < j \leq k \\ b_i = b_j  \\ j \neq i+1 \\ (i,j) \neq (0,k) }}  (y_j - y_i)  $$
where the variable~$x_i$ ranges over the values~$\{ c : (c,a_i) \in S\}$ for each~$i$.

We claim that an assignment of the variables makes ${\bm{x_a}}$ a sequencing of~$S$ if and only if  this polynomial is nonzero. 

The first product compares elements in the sequencing: it contains a factor that is zero if and only if $(x_i,a_i) = (x_j,a_j)$ for some~$i,j$.  The second product does a similar task for the partial sums when first two conditions on the product are considered.  The second two conditions reduce the degree of the polynomial, which is generally a positive effect as it makes it easier to meet the degree condition in the Non-Vanishing Corollary and reduces the amount of work required to calculate a coefficient.  The condition $j \neq i+1$ is permissible, because we know that  $(y_i,a_i) \neq (y_{i+1},b_{i+1})$ by the assumption that $(0,0) \not\in S$.  The condition $(i,j) \neq (0,k)$ is permissible as we know that $(y_0,b_0) = (0,0)$ and $(y_k, b_k) = \Sigma S$: the polynomial will detect linear or rotational sequencings according to whether $\Sigma S$ is nonzero or zero respectively.  

For each~$i$, the number of possible values for~$x_i$ is $| \{ c : (c,a_i) \in S\} |$.  Therefore the bounding monomial for use with the Non-Vanishing Corollary is given by $x_1^{|C_1| - 1} \cdots x_k^{|C_k| - 1}$, where $C_i$ is the set of elements in~$S$ that are in the coset~$\overline{ (0,a_i) }$.

\begin{exa}\label{ex:32}
Let~$G = \Z_p \times \Z_2$ with $p > 2$ and $p$ prime.  Suppose~$S \subseteq G \setminus \{ (0,0) \}$ with $|S|=5$ and of type~$(3,2)$.   The sequence $\bm{a} = (0,1,0,0,1)$ has partial sums $(0,0,1,1,1,0)$ and so is a quotient sequencing of~$\pi_2(S)$ with respect to~$p$.
We desire a sequencing of~$S$ of the form 
$$\left( (x_1, 0), (x_2,1), (x_3,0), (x_4,0), (x_5,1) \right).$$
The polynomial is
\begin{eqnarray*}
p_{\bm{a}} &=&  (x_3-x_1)(x_4-x_1)(x_4-x_3)(x_5-x_2)(y_5 - y_1)(y_4 - y_3)  \\
     & = & (x_3-x_1)(x_4-x_1)(x_4-x_3)(x_5-x_2)(x_2+ x_3+x_4+x_5)(x_3+x_4). \\
\end{eqnarray*}

To apply the Non-Vanishing Corollary we need a monomial of this polynomial which divides the bounding monomial $x_1^2x_2x_3^2x_4^2x_5$ with a nonzero coefficient.  One such is $x_1^2x_3^2x_4x_5$, which has coefficient~$-1$.   Hence whenever~$S$ has this form it has a 
 sequencing.  

Noting that $G \cong \Z_{2p}$, we can rephrase this as showing that a subset~$S$ of ~$\Z_{2p} \setminus \{0\}$ has a sequencing whenever~$|S|=5$ and $S$ has exactly~3 even elements.
\end{exa}

When~$t=1$ we always have $a_i = 0 = a_j$ and $b_i = 0 = b_j$ and the polynomial~$p_{\bm{ a}}$ reduces to the one used in~\cite{HOS19} to investigate Alspach's Conjecture in~$\Z_p$ for $p$ prime:
$$p_{\bm{a}} = \prod_{\substack{1 \leq i < j \leq k }}  (x_j - x_i)  
                      \prod_{\substack{0 \leq i < j \leq k \\ j \neq i+1 \\ (i,j) \neq (0,k) }}  (y_j - y_i).  $$
As noted in the introduction, that paper was concerned only with the case $\Sigma S \neq 0$, but the above discussion shows that the same calculations also prove Theorem~\ref{th:known}.3 when $\Sigma S = 0$.

The polynomial method approach may also be used for Conjecture~\ref{conj:cmpp}.    While the same polynomial suffices, we can reduce the degree slightly by considering this alternative:
$$q_{\bm{a}} = \prod_{\substack{1 \leq i < j \leq k \\ a_i = a_j }}  (x_j - x_i)  
                      \prod_{\substack{0 \leq i < j \leq k \\ b_i = b_j  \\ j \not\in \{i+1,i+2\} \\ (i,j) \neq (0,k) }}  (y_j - y_i).  $$
The difference compared to~$p_{\bm{ a}}$ is that we have removed factors of the form $(y_{i+2} - y_i)$.  In the context of Conjecture~\ref{conj:cmpp} we know that a factor of this form is nonzero when $y_i$ and $y_{i+2}$ are in the same coset because $(x_{i+1},a_{i+1}) \neq - (x_i, a_i)$.    However, in the next section all computations use $p_{\bm{a}}$ rather than~$q_{\bm{a}}$.

\section{Computational Results}\label{sec:results} 

We begin with the case~$t=1$; that is, groups of prime order.

\begin{thm}\label{th:k11_prime}
Let $p$ be prime and let $S \subseteq \Z_p \setminus \{ 0 \}$ with $|S| \in  \{11,12 \}$.   Then $S$ is sequenceable.  
\end{thm}

\begin{proof}
First, consider the case~$|S|=11$.
The polynomial $p_{\bm{a}}$ for this situation has degree~109 and to use the Non-Vanishing Corollary we require a monomial that divides the bounding monomial $x_1^{10}x_2^{10}\cdots x_{11}^{10}$ of degree~110, which gives eleven possible monomials.  Over the integers, the coefficient on $x_1^{9}x_2^{10}\cdots x_{11}^{10}$ is 
$$-18128730243333160 = -  2^3\cdot 5\cdot 11\cdot 3019\cdot 13647452681$$
and the coefficient on $x_1^{10}x_2^{9}x_3^{10}  \cdots x_{11}^{10}$ is 
$$ -46383022877233608 = - 2^3 \cdot 3^2 \cdot 644208651072689$$
(in each case the right-hand side gives the prime factorization).

The two integers have no odd prime factors in common.  Therefore, for any odd prime~$p$ (in particular, for any prime greater than~11, which is the current concern) the coefficient in~$\Z_p$ is nonzero for at least one of these two monomials.  The Non-Vanishing Corollary gives the result.

We use the same approach for~$|S|=12$.  The polynomial~$p_{\bm{a}}$ now has degree~131 and the prime factorizations of the coefficients on the monomials $x_1^{10}x_2^{11}\cdots x_{12}^{11}$ and  $x_1^{11}x_2^{10}x_3^{11} \cdots x_{12}^{11}$ are
$$2^4 \cdot 3 \cdot 29 \cdot 12953077208391719881 \  \text{ and }  \  2^3 \cdot 3 \cdot 277 \cdot 1901 \cdot 786640832519761$$
and the result follows.
\end{proof}

The program used for the computation used in the proof of Theorem~\ref{th:k11_prime} was independent of those used in~\cite{HOS19}.  This new program recalculated the coefficients obtained in that paper, obtaining the same results.

Moving on to $t>1$, the method described in the previous section divides the proof into a case for each type.  For each case, the first step of the process is to find a quotient sequencing that matches the type.  There are typically many of these.   We tend to choose 
a quotient sequencing whose partial sums are distributed among the elements of~$\Z_t$ as evenly as possible, as this both reduces the degree of $p_{\bm{a}}$ and minimizes the smallest value of~$p$ that it is a quotient sequencing with respect to.

For many types of~$S$, we find a quotient sequencing for which the degree of $p_{\bm{a}}$ is significantly lower than the degree of the bounding monomial.  This gives the scope to prove slightly stronger results with less computation.  

Suppose that we have set a quotient sequencing $\bm{a} = (a_1, \ldots, a_k)$.  Take $\ell$ with $1 \leq \ell \leq k$ and assume $x_{\ell} = c$ for some constant~$c$ such that $(c, a_\ell) \in S$.   Then we define the following non-homogeneous polynomial that is nonzero if and only if there is a sequencing of~$S$ that has quotient sequencing~$\bm{a}$ and with $(c, a_\ell)$ in position~$\ell$
$$p'_{\bm{a}} = \prod_{\substack{1 \leq i < j \leq k \\ a_i = a_j \\ i, j \neq \ell}}  (x_j - x_i)  
                        \prod_{\substack{0 \leq i < j \leq k \\ b_i = b_j \\ j \neq i+1 \\ (i,j) \neq (0,k) }}  (y_j - y_i)$$
As we are interested in terms of maximum degree, we may replace factors in $p'_{\bm{a}}$ that include $c$ as a summand by the same factor with $c$ removed.

This process reduces the complexity of the initial polynomial $p_{\bm{a}}$, which is an advantage for computing coefficients.  It also reduces the degree of the bounding monomial as there is one fewer element available in the coset $\overline{(0,a_\ell)}$.  This can mean that the Non-Vanishing Corollary does not apply and that  we therefore cannot make this step.   

The process may be repeated by choosing multiple elements to fix as arbitrary constants.  However, we must ensure that there are no relations between the constants and their positions that could lead to the polynomial returning a nonzero for a non-sequencing, which could happen if a factor is a linear combination of constants with no variables.  We achieve this by never fixing both $x_{\ell}$ and $x_{\ell'}$ with $\ell' \in \{\ell - 1, \ell + 1\}$,  which means that every factor retains at least one variable. 

In general, we keep fixing elements until there are no more that we may fix without becoming unable to apply the Non-Vanishing Corollary.

\begin{exa}
Let~$G = \Z_p \times \Z_2$ with $p > 3$ and $p$ prime.  Suppose~$S \subseteq G \setminus \{ (0,0) \}$ with $|S|=7$ and of type~$(5,2)$.   The sequence $\bm{a} = (0,0,1,0,0,0,1)$ has partial sums $(0,0,0,1,1,1,1,0)$ and so is a quotient sequencing of~$\pi_2(S)$ with respect to~$p$.
We desire a sequencing of~$S$ of the form 
$$\left( (x_1, 0), (x_2,0), (x_3,1), (x_4,0), (x_5,0), (x_6, 0), (x_7, 1) \right)$$
with partial sums
$$\left( (y_0, 0), (y_1, 0), (y_2,0), (y_3,1), (y_4, 1) (y_5,1), (y_6, 1), (y_7, 0) \right).$$
The polynomial is
\begin{eqnarray*}
p_{\bm{a}} &=&  (x_2-x_1)(x_4-x_1) (x_5 - x_1) (x_6 - x_1) (x_4-x_2)(x_5-x_2) (x_6 - x_2) (x_7 - x_3)  \\ 
     && (x_5-x_4) (x_6-x_4) (x_6-x_5) (y_2 - y_0)(y_7 - y_1) (y_7 - y_2) (y_5-y_3) (y_6 - y_3) (y_6 - y_4)  \\
     & = & (x_2-x_1)(x_4-x_1) (x_5 - x_1) (x_6 - x_1) (x_4-x_2)(x_5-x_2) (x_6 - x_2) (x_7 - x_3)  \\
     &&  (x_5-x_4) (x_6-x_4) (x_6-x_5) (x_1 + x_2) (x_2 + x_3 + x_4 + x_5 + x_6 + x_7)\\
     && (x_3 + x_4 + x_5 + x_6 + x_7) (x_4 + x_5) (x_4 + x_5 + x_6) (x_5 + x_6) .
\end{eqnarray*}
Since the polynomial $p_{\bm{a}}$ has degree $17$ and the bounding monomial $x_1^4 x_2^4 x_3 x_4^4 x_5^4 x_6^4 x_7$ has degree $22$, we can fix $x_3 = c_1$ and $x_6 = c_2$, where $(c_1,  1),  (c_2,  0) \in S$,  without violating the constraint on the degree of the polynomial which has to be less than or equal to the degree of the bounding monomial to satisfy the hypotheses of the Non-Vanishing Corollary. Therefore we get the following simplified polynomial of degree $12$
\begin{eqnarray*}
p'_{\bm{a}} &=& (x_2-x_1)(x_4-x_1) (x_5 - x_1) (x_4-x_2)(x_5-x_2) (x_5-x_4) (x_1 + x_2) (x_4 + x_5) \\ 
     && (x_2  + c_1 + x_4 + x_5  + c_2 +  x_7) (c_1 + x_4 + x_5 + c_2 +  x_7) (x_4 + x_5 + c_2) (x_5 + c_2).
\end{eqnarray*}

To apply the Non-Vanishing Corollary we need a monomial of this polynomial which divides the new bounding monomial $x_1^3 x_2^3 x_4^3 x_5^3$ with a nonzero coefficient.  Since the degree of $p'_{\bm{a}}$ is equal to the degree of the new bounding monomial,  the only feasible monomial is  $x_1^3 x_2^3 x_4^3 x_5^3$, which has coefficient~$-2$.   Hence whenever~$S$ has this form it has a sequencing. 
\end{exa}

The following result completes the proof of Theorem~\ref{th:main}.

\begin{thm}\label{th:composite}
Let $n = pt$ with~$p$ prime.  Then subsets~$S$ of size~$k$ of~$\Z_n \setminus \{ 0 \}$ are sequenceable in the following cases:
\begin{enumerate}
\item  $k \leq 11$ and $t \in \{ 2,3,4, 5\}$,
\item  $k=12$ and $t \in \{2,3, 4\}$, 
\item  $k= 13$ and $t \in \{2,3\}$, provided~$S$ contains at least one element not in the subgroup of order~$p$,
\item  $k=14$ and $t = 2$,   provided~$S$ contains at least one element not in the subgroup of order~$p$,
\item $k=15$ and $t=2$, provided~$S$ does not contain exactly~$0$,~$1$,~$2$ or~$15$ elements of the subgroup of order~$p$.
\end{enumerate}
\end{thm}

\begin{proof}
We can suppose that $k \geq 10$ since by \cite{AL20} we know that the subsets $S$ of size $k \leq 9$ in an arbitrary abelian group are sequenceable. 
Then,  in all of the cases stated in the theorem, we can use the Non-Vanishing Corollary since $\mathbb{Z}_p \times \mathbb{Z}_t \cong \mathbb{Z}_{pt}$ except for $p=t=5$ which has been treated separately.  In that case we have checked computationally that each subset of size $10$ and $11$ of $\mathbb{Z}_{25} \setminus \{0\}$ is sequenceable (the code is available from the ArXiv page for this paper).  
All of the other results are obtained using the Python framework SageMath \cite{SageMath} mainly because it has efficient libraries to handle multivariate polynomials.  The polynomial multiplication was carried out by multiplying pairs of factors and then all the resulting terms together.  The specific order of the products is described by the pseudo-code in Algorithm \ref{alg:polmul}.
\begin{algorithm}[ht!]
\small
\caption{Polynomial multiplication}\label{alg:polmul}
\begin{algorithmic}
\Require 
\State $G = \mathbb{Z}_p \times \mathbb{Z}_t$,  $p$ prime, $t \geq 1$ and $p$ coprime with $t$
\State $S \subseteq G \setminus \{ (0,0) \}$, $|S| = k$
\State $\bm{a} = (a_1, a_2, \ldots, a_k)$ with partial sums $\bm{b} = (b_0, b_1, \ldots, b_k)$
\State $p \gets 1$
\For{$1 \leq i < j \leq k$}
\State $f \gets 1$
\If{$a_i = a_j$}
\State $f \gets f \cdot (x_j - x_i)$
\EndIf
\If{$b_{i-1} = b_j$ \textbf{and} $(i-1, j) \neq (0, k)$ }
\State $f \gets f \cdot (x_i + \ldots + x_j)$
\EndIf
\State $p \gets p \cdot f$
\EndFor
\State \Return $p$
\end{algorithmic}
\end{algorithm}

In addition, after each multiplication, we only keep the terms that divide the bounding monomial.  If we are searching for the coefficient of a specific monomial then we can also lower bound the exponents of each term because at some point of the computation the degrees of the variables $x_i$'s cannot be too small.  For the sake of readability, the pseudo-code reported in Algorithm \ref{alg:polmul} does not include the restrictions on the exponents of each monomial. However, the code associated with the final version of the algorithm is publicly available at the ArXiv page for this paper.  Moreover, all the tables reporting the monomials' coefficients for each case listed in Theorem \ref{th:composite} can be found at the same location.  All of the computations were completed in less than 5 days on a PC with a 4.6 GHz AMD Ryzen 9 processor and 128 GB of RAM.
\end{proof}

As an example of the process of the proof of Theorem~\ref{th:composite}, Table~\ref{tab:10_2} contains the required monomials and their coefficients for the case~$t=2$ and~$|S| = 10$.

\small
\begin{longtable}{lllll}
    \caption{Monomials and their coefficients sufficient for the proof of Theorem~\ref{th:composite} in the case $|S|=10$ and~$t=2$.}\label{tab:10_2}\\
	\hline
	$\boldsymbol{\lambda}$ & $\mathbf{a}$ & deg & monomial/s & coefficient/s \\
	\hline
	\endfirsthead
	\hline
	$\boldsymbol{\lambda}$ & $\mathbf{a}$ & deg & monomial/s & coefficient/s \\
	\hline
	\endhead
$(10, 0)$ & \begin{tabular}{@{}l@{}} $(0, 0, 0, 0, 0, 0, 0, 0, 0, 0)$ \end{tabular} & \begin{tabular}{@{}l@{}} 89 \end{tabular} & \begin{tabular}{@{}l@{}} $x_{1}^{8}x_{2}^{9}x_{3}^{9}x_{4}^{9}x_{5}^{9}x_{6}^{9}x_{7}^{9}x_{8}^{9}x_{9}^{9}x_{10}^{9}$ \\ $x_{1}^{9}x_{2}^{8}x_{3}^{9}x_{4}^{9}x_{5}^{9}x_{6}^{9}x_{7}^{9}x_{8}^{9}x_{9}^{9}x_{10}^{9}$ \end{tabular} & \begin{tabular}{@{}l@{}} $2^5 \cdot 7 \cdot 11^2 \cdot 21966239$ \\ $2 \cdot 13 \cdot 211 \cdot 256046627$ \end{tabular}\\ \hline \\ 
$(9, 1)$ & \begin{tabular}{@{}l@{}} $(0, 0, 0, 0, 0, 1, 0, 0, 0, 0)$ \end{tabular} & \begin{tabular}{@{}l@{}} 52 \end{tabular} & \begin{tabular}{@{}l@{}} $x_{2}^{2}x_{3}^{4}x_{4}^{7}x_{5}^{8}x_{7}^{7}x_{8}^{8}x_{9}^{8}x_{10}^{8}$ \end{tabular} & \begin{tabular}{@{}l@{}} $-1 \cdot 2^2$ \end{tabular}\\ \hline \\ 
$(8, 2)$ & \begin{tabular}{@{}l@{}} $(0, 1, 0, 0, 0, 0, 1, 0, 0, 0)$ \end{tabular} & \begin{tabular}{@{}l@{}} 45 \end{tabular} & \begin{tabular}{@{}l@{}} $x_{1}x_{3}x_{4}^{7}x_{5}^{7}x_{6}^{7}x_{7}x_{8}^{7}x_{9}^{7}x_{10}^{7}$ \end{tabular} & \begin{tabular}{@{}l@{}} $-1 \cdot 2 \cdot 3 \cdot 7$ \end{tabular}\\ \hline \\ 
$(7, 3)$ & \begin{tabular}{@{}l@{}} $(0, 0, 0, 0, 1, 0, 0, 0, 1, 1)$ \end{tabular} & \begin{tabular}{@{}l@{}} 42 \end{tabular} & \begin{tabular}{@{}l@{}} $x_{2}^{6}x_{3}^{6}x_{4}^{6}x_{5}^{2}x_{6}^{6}x_{7}^{6}x_{8}^{6}x_{9}^{2}x_{10}^{2}$ \end{tabular} & \begin{tabular}{@{}l@{}} $-1 \cdot 2 \cdot 3 \cdot 7$ \end{tabular}\\ \hline \\ 
$(6, 4)$ & \begin{tabular}{@{}l@{}} $(0, 0, 0, 1, 0, 0, 0, 1, 1, 1)$ \end{tabular} & \begin{tabular}{@{}l@{}} 39 \end{tabular} & \begin{tabular}{@{}l@{}} $x_{1}^{5}x_{2}^{5}x_{3}^{5}x_{4}^{3}x_{5}^{5}x_{6}^{5}x_{7}^{3}x_{8}^{3}x_{9}^{3}x_{10}^{2}$ \end{tabular} & \begin{tabular}{@{}l@{}} $2\cdot5$ \end{tabular}\\ \hline \\ 
$(5, 5)$ & \begin{tabular}{@{}l@{}} $(0, 0, 0, 1, 0, 0, 1, 1, 1, 1)$ \\ $(0, 1, 0, 1, 0, 1, 0, 1, 0, 1)$ \end{tabular} & \begin{tabular}{@{}l@{}} 40 \\ 40 \end{tabular} & \begin{tabular}{@{}l@{}} $x_{1}^{4}x_{2}^{4}x_{3}^{4}x_{4}^{4}x_{5}^{4}x_{6}^{4}x_{7}^{4}x_{8}^{4}x_{9}^{4}x_{10}^{4}$ \\ $x_{1}^{4}x_{2}^{4}x_{3}^{4}x_{4}^{4}x_{5}^{4}x_{6}^{4}x_{7}^{4}x_{8}^{4}x_{9}^{4}x_{10}^{4}$ \end{tabular} & \begin{tabular}{@{}l@{}} $2^2 \cdot 157$ \\ $5 \cdot 19 \cdot 41 \cdot 83$ \end{tabular}\\ \hline \\ 
$(4, 6)$ & \begin{tabular}{@{}l@{}} $(0, 1, 0, 1, 1, 1, 1, 0, 1, 0)$ \end{tabular} & \begin{tabular}{@{}l@{}} 41 \end{tabular} & \begin{tabular}{@{}l@{}} $x_{1}^{2}x_{2}^{5}x_{3}^{3}x_{4}^{5}x_{5}^{5}x_{6}^{5}x_{7}^{5}x_{8}^{3}x_{9}^{5}x_{10}^{3}$ \\ $x_{1}^{3}x_{2}^{4}x_{3}^{3}x_{4}^{5}x_{5}^{5}x_{6}^{5}x_{7}^{5}x_{8}^{3}x_{9}^{5}x_{10}^{3}$ \end{tabular} & \begin{tabular}{@{}l@{}} $2^4 \cdot 3 \cdot 5 \cdot 13$ \\ $2 \cdot 3 \cdot 463$ \end{tabular}\\ \hline \\ 
$(3, 7)$ & \begin{tabular}{@{}l@{}} $(0, 0, 1, 0, 1, 1, 1, 1, 1, 1)$ \end{tabular} & \begin{tabular}{@{}l@{}} 46 \end{tabular} & \begin{tabular}{@{}l@{}} $x_{2}^{2}x_{3}^{6}x_{4}^{2}x_{5}^{6}x_{6}^{6}x_{7}^{6}x_{8}^{6}x_{9}^{6}x_{10}^{6}$ \end{tabular} & \begin{tabular}{@{}l@{}} $-1 \cdot 2^3 \cdot 3^2$ \end{tabular}\\ \hline \\ 
$(2, 8)$ & \begin{tabular}{@{}l@{}} $(0, 1, 0, 1, 1, 1, 1, 1, 1, 1)$ \end{tabular} & \begin{tabular}{@{}l@{}} 51 \end{tabular} & \begin{tabular}{@{}l@{}} $x_{1}x_{2}x_{3}x_{4}^{6}x_{5}^{7}x_{6}^{7}x_{7}^{7}x_{8}^{7}x_{9}^{7}x_{10}^{7}$ \\ $x_{1}x_{3}x_{4}^{7}x_{5}^{7}x_{6}^{7}x_{7}^{7}x_{8}^{7}x_{9}^{7}x_{10}^{7}$ \end{tabular} & \begin{tabular}{@{}l@{}} $-1 \cdot 2 \cdot 1277$ \\ $-1 \cdot 2 \cdot 17^2$ \end{tabular}\\ \hline \\ 
$(1, 9)$ & \begin{tabular}{@{}l@{}} $(1, 0, 1, 1, 1, 1, 1, 1, 1, 1)$ \end{tabular} & \begin{tabular}{@{}l@{}} 60 \end{tabular} & \begin{tabular}{@{}l@{}} $x_{1}^{2}x_{3}^{2}x_{4}^{8}x_{5}^{8}x_{6}^{8}x_{7}^{8}x_{8}^{8}x_{9}^{8}x_{10}^{8}$ \\ $x_{1}^{2}x_{3}^{3}x_{4}^{7}x_{5}^{8}x_{6}^{8}x_{7}^{8}x_{8}^{8}x_{9}^{8}x_{10}^{8}$ \end{tabular} & \begin{tabular}{@{}l@{}} $2 \cdot 17^2$ \\ $2^2 \cdot 647$ \end{tabular}\\ \hline \\ 
$(0, 10)$ & \begin{tabular}{@{}l@{}} $(1, 1, 1, 1, 1, 1, 1, 1, 1, 1)$ \end{tabular} & \begin{tabular}{@{}l@{}} 69 \end{tabular} & \begin{tabular}{@{}l@{}} $x_{1}^{2}x_{2}^{2}x_{3}^{4}x_{4}^{7}x_{5}^{9}x_{6}^{9}x_{7}^{9}x_{8}^{9}x_{9}^{9}x_{10}^{9}$ \\ $x_{1}^{2}x_{2}^{2}x_{3}^{4}x_{4}^{9}x_{5}^{7}x_{6}^{9}x_{7}^{9}x_{8}^{9}x_{9}^{9}x_{10}^{9}$ \end{tabular} & \begin{tabular}{@{}l@{}} $2 \cdot 3 \cdot 733$ \\ $2^5 \cdot 3^2 \cdot 5$ \end{tabular}\\ \hline
\end{longtable}
\normalsize

\section{To Infinity and Back}\label{sec:infinity}

Here we will say that a finite subset $S$ of an abelian group $G$
is non-zero sum if $\Sigma S\not= 0_G$.

We recall a definition and lemma from~\cite{CP20}.
First, given a subset $S$ of an abelian group $G$, we define the set
$$\Upsilon(S)=S\cup\Delta(S)\cup\left\{\Sigma S \right\}.$$
 
Then we can paraphrase~\cite[Lemma 2.1]{CP20} as
\begin{lem}\label{cambiogruppo}
Let $G_1$ and $G_2$ be abelian groups.
Given a non-zero sum subset $S$ of $G_1\setminus \{0_{G_1}\}$ of size $k$, suppose there exists an homomorphism
$\varphi: G_1\to G_2$ such that $\ker(\varphi)\cap \Upsilon(S)=\emptyset$.
Then $\varphi(S)$ is a non-zero sum subset of $G_2\setminus\{0_{G_2}\}$ of size $k$ and the subset $S$ is sequenceable whenever $\varphi(S)$ is sequenceable.
\end{lem}

In the following, we consider abelian groups of the form $G\times H$ where  $H=\{h_0=0_H,h_1,\dots,h_{t-1}\}$ is finite. Given a multiset $S$ of elements of $G\times H$, we can define the {\em type} of~$S$ also when $G$ is a generic group. It is the sequence $\bm{\lambda} = (\lambda_0, \ldots \lambda_{t-1} )$, where $\lambda_i$ is the number of times that $h_i$ appears in $\pi_2(S)$ and $\pi_2$ is the projection over $H$. Note that two multisets $S_1$ of $G_1\times H$ and $S_2$ of $G_2\times H$ can have the same type even though $G_1$ is not equal to $G_2$.

\begin{prop}\label{prop:Z}
Let $H$ be a finite abelian group and suppose that, for infinitely many primes $p$, any non-zero sum subset of a given type $\bm{\lambda} = (\lambda_0, \ldots \lambda_{t-1} )$ of  $\Z_p\times H\setminus \{0_{\Z_p\times H}\}$ is sequenceable. Then also any non-zero sum subset of type $\bm{\lambda}$ of $\Z\times H\setminus \{0_{\Z\times H}\}$ is sequenceable.
\end{prop}

\begin{proof}
Consider a non-zero sum subset $S$ of $\Z\times H\setminus \{0_{\Z\times H}\}$ of type $\bm{\lambda}$. Let $p>\max\limits_{z\in \Upsilon(S)} |\pi_{\Z}(z)|$ be a prime such that any non-zero sum subsets of $\Z_p\times H\setminus \{0_{\Z_p\times H}\}$ of type $\bm{\lambda}$ is sequenceable.
Then, $\Upsilon(S)$ and the kernel of the map $\pi_p\times Id: \Z\times H\to \Z_p\times H$ are disjoint sets.
Therefore, due to Lemma \ref{cambiogruppo}, $(\pi_p\times Id)(S)$ is also non-zero sum and, since $\pi_p\times Id$ is the identity on $H$ and $S$ is of type $\bm{\lambda}$, also $(\pi_p\times Id)(S)$ is of type $\bm{\lambda}$. Then, it follows from Lemma \ref{cambiogruppo}, that also $S$ is sequenceable.
\end{proof}

We now consider the free abelian group $\Z^n$ of rank $n$.

\begin{prop}\label{prop:Zn}
Let $H$ be a finite abelian group and suppose that any non-zero sum subset of a given type $\bm{\lambda} = (\lambda_0, \ldots \lambda_{t-1} )$ is sequenceable.
Then also any non-zero sum subset of type $\bm{\lambda}$ of $\Z^n\times H\setminus\{0_{\Z^n\times H}\}$ (for any $n\geq 2$) is sequenceable.
\end{prop}

\begin{proof}
Fix a subset $S$ of $\Z^{n}\times H\setminus \{0_{\Z^{n}\times H}\}$ of type $\bm{\lambda}$, and set $B=\pi_{\Z^n}(\Upsilon(S))$.
Given an integer
$$M>\max\limits_{(z_1,\dots,z_{n}) \in B}\max\limits_{j\in [1,n]} n|z_j|,$$
we define the
homomorphism $\varphi: \Z^{n} \to \Z$, as follows:
$$\varphi(z_1,\dots,z_{n})=\sum\limits_{i=1}^{n} z_i M^{i-1}.$$
Then we consider the homomorphism $\varphi\times Id: \Z^{n}\times H \to \Z\times H$.
Because of the choice of $M$, the subset $\Upsilon(S)$ and the
kernel of $\varphi\times Id$ are disjoint. Namely, suppose that there exist $b \in \Z^n$ and $h\in H$ such
that $(b,h)\in \Upsilon(S)\cap \ker(\varphi\times Id)$ which clearly implies that $b\in B$ and $h=0_H$. Set $b=(b_1,b_2,\ldots,b_n)$, we can assume that $b_{s_1}, \ldots,b_{s_c}$ are all nonnegative integers while
$b_{t_1}, \ldots,b_{t_d}$ are all negative integers.
Then, we can write
$$\sum_{j=1}^c b_{s_j} M^{s_j-1} =\sum_{j=1}^d -b_{t_j} M^{t_j-1}.$$
We can look at the two sides of this equality as two expansions in base $M$ of the same nonnegative integer,
since the coefficients $b_{s_1},\ldots,b_{s_c}, -b_{t_1},\ldots,-b_{t_d}$ all belong to the set $[0,M-1]$.
The uniqueness of such expansion implies that all these coefficients are zero, i.e., that $b=0_{\Z^n}$. We recall that $(\varphi\times Id)(b,h)=0_{\Z\times H}$ also implies that $h=0_H$, and hence we obtain that $(b,h)=0_{\Z^{n}\times H}$ but this is in contradiction with the assumptions that both $S$ and $\Delta(S)$ does not contain zero and that $S$ is non-zero sum.

Therefore, due to Lemma \ref{cambiogruppo}, $(\varphi\times Id)(S)$ is also non-zero sum and, since $\varphi\times Id$ is the identity on $H$, it is of type $\bm{\lambda}$. Then, it follows from Lemma \ref{cambiogruppo} that also~$S$ is sequenceable.
\end{proof}

From the previous proposition, we deduce this result.

\begin{thm}\label{torsionfree}
Let $H$ be a finite abelian group and suppose that, for infinitely many primes $p$, any non-zero sum subset of a given type $\bm{\lambda} = (\lambda_0, \ldots \lambda_{t-1} )$ of $\Z_p\times H\setminus\{0_{\Z_p\times H}\}$ is sequenceable. Then also any non-zero sum subset of $G\times H\setminus\{0_{G\times H}\}$ of type $\bm{\lambda}$ is sequenceable provided that $G$ is torsion-free and abelian.
\end{thm}

\begin{proof}
Let $S$ be a non-zero sum subset of type $\bm{\lambda}$ of $G\times H\setminus\{0_{G\times H}\}$ and let $k=\lambda_0+\lambda_1+\cdots+\lambda_{t-1}$ be its size. Denote by $K$ the subgroup of $G\times H$ generated by $S$.
We can apply to $\pi_G(K)$ the structure theorem for finitely generated abelian groups, obtaining that $\pi_G(K)$
is isomorphic to a subgroup of $\Z^{k}$. So, we can view $S$ as a non-zero sum subset of type $\bm{\lambda}$ of $\Z^k\times H\setminus \{0_{\Z^k\times H}\}$.
Since, by hypothesis, we are assuming that, for infinitely many primes $p$ any non-zero sum subset of type $\bm{\lambda}$ of $\Z_p\times H\setminus \{0_{\Z_p\times H}\}$ is sequenceable, by Propositions \ref{prop:Z} and \ref{prop:Zn},
also any non-zero sum subset of type $\bm{\lambda}$ of $\Z^k\times H\setminus\{0_{\Z^k\times H}\}$ is sequenceable and hence also $S$ it is.
\end{proof}
\begin{cor}
Let $H$ be a finite abelian group, $k$ be a positive integer, and suppose that, for infinitely many primes $p$, any non-zero sum subset of size $k$ of $\Z_p\times H\setminus\{0_{\Z_p\times H}\}$ is sequenceable. Then also any non-zero sum subset of size $k$ of $G\times H\setminus \{0_{G\times H}\}$ is sequenceable provided that $G$ is torsion-free and abelian.
\end{cor}
Given an element $g$ of an abelian group $G$, we denote by $o(g)$ the cardinality of the cyclic subgroup $\langle g
\rangle$ generated by $g$. Furthermore, we set
$$\vartheta(G)=\min_{0_G\neq g \in G} o(g).$$
In the following, developing the notations of the previous sections, given a set $S$ of size $k$ and an ordering $\bm{x}=(x_1,\dots,x_k)$ of $S$, we denote by $y_i(\bm{x})$ the partial sum $x_1+x_2+\dots+x_i$.

As a consequence of Theorem \ref{torsionfree}, we can prove the following result.
\begin{thm}\label{thm:as}
Let $H$ be a finite abelian group and suppose that, for infinitely many primes $p$, any non-zero sum subset of $\Z_p\times H\setminus \{0_{\Z_p\times H}\}$ of a given type $\bm{\lambda} = (\lambda_0, \ldots \lambda_{t-1} )$ is sequenceable. 
Then, there exists a positive integer $N(\bm{\lambda})$ such that any non-zero sum subset of type $\bm{\lambda}$ of $G\times H\setminus\{0_{G\times H}\}$ is sequenceable provided that $G$ is abelian and $\vartheta(G)>N(\bm{\lambda})$.
\end{thm}

\begin{proof}
Let us suppose, for sake of contradiction, that such $N(\bm{\lambda})$ does not exist.
It means that, for any positive integer $M$, there exists an abelian group $G_M$ such that $\vartheta(G_M)>M$ and
there exists a non-zero sum subset of type $\bm{\lambda}$,
$S_M=\{x_{M,1},x_{M,2},\dots,x_{M,k}\}$ of $G_M\times H\setminus\{0_{G_M\times H}\}$ that is not sequenceable. Here $k=\lambda_0+\lambda_1+\dots+\lambda_{t-1}$.
Therefore, for any permutation $\omega$ of
$[1,k]$, defined the ordering $$\bm{x}_{M,\omega}=(x_{M,\omega(1)},x_{M,\omega(2)},\dots,x_{M,\omega(k)})$$ of $S_M$ there exists a pair $(s,t)$, with $s,t \in [1,k]$ and $s\not=t$, such that $y_s(\bm{x}_{M,\omega})=y_t(\bm{x}_{M,\omega})$.
Choosing for each $\omega\in \Sym(k)$ one of these pairs, for any positive integer $M$ we can define the function
$f_M: \Sym(k) \to [1,k]\times [1,k]$ that maps
$\omega$ into this chosen pair.
Since there are only finitely many maps from $\Sym(k)$ to $[1,k]\times [1,k]$, there exists an infinite sequence
$(M_1,\ldots, M_{\ell},\ldots)$ such that $f_{M_1}=f_{M_{\ell}}$ for all $\ell\geq 1$.
Moreover, since $H$ is a finite group, there exists an infinite subsequence $(M_{\sigma(1)},\ldots, M_{\sigma(\ell)},\ldots)$ of $(M_1,\ldots, M_\ell,\ldots)$ such that the projection of $S_{M_{\sigma(\ell)}}=\{x_{M_{\sigma(\ell)},1},x_{M_{\sigma(\ell)},2},\ldots,x_{M_{\sigma(\ell)},k}\}$ on $H$ does not depend on $\ell$, and thus, for any $j\in [1,k]$ and $\ell\in \mathbb{N}$, we have $\pi_H(x_{M_{\sigma(1)},j})=\pi_H(x_{M_{\sigma(\ell)},j})$.

Let us consider the group
$G=\bigtimes\limits_{i=1}^{\infty} G_{M_{\sigma(i)}}$ and the following equivalence relation on $G$.
Given $a=(a_i),b=(b_i)\in G$, we set $a\approx b$ whenever $a_i\neq b_i$ only on a finite number of indices $i$.
Since the equivalence class $[0]$ consists of the elements $(a_i)$ of $G$ that are nonzero on a
finite number of coordinates $a_i$, and so it is a subgroup of $G$,
the quotient set $G'=G/\approx$ is still an abelian group.

Now we want to prove that $G'$ is torsion-free. Let us suppose for sake of contradiction that there exists
an element $[a]\neq0_{G'}=[0]$ in $G'$ of finite order, say $n$. Let $\pi_{\ell}: G\to G_{M_{\sigma(\ell)}}$ be the canonical projection on $G_{M_{\sigma(\ell)}}$.
For any $\ell$ such that $M_{\sigma(\ell)}>n$, either $\pi_{\ell}(a)=0_{G_{M_{\sigma(\ell)}}}$ or
we have $n\cdot \pi_{\ell}(a)\not=0_{G_{M_{\sigma(\ell)}}}$. However,
since $n\cdot [a]=[0]$ in $G'$ and due to the definition of $\approx$,
we should have $n\cdot \pi_{\ell}(a)=0_{G_{M_{\sigma(\ell)}}}$ for $\ell$ large enough. It follows that $\pi_{\ell}(a)$ is eventually zero but
this is a contradiction since $[a]$ is nonzero. Therefore $G'$ is torsion-free.

Now we consider the following subset $S$ of $G'\times H$
$$S=\{([z_1], \pi_H(x_{M_{\sigma(1)},1})),([z_2],\pi_H(x_{M_{\sigma(1)},2})),\dots,([z_k],\pi_H(x_{M_{\sigma(1)},k}))\}$$
where $z_j\in G$ is such that $\pi_{\ell}(z_j)=\pi_{G_{M_{\sigma(\ell)}}}(x_{M_{\sigma(\ell)},j})$.
Note that, since the projection of $S_{M_{\sigma(\ell)}}$ on $H$ does not depend on $\ell$, $S$ is a non-zero sum subset of $G'\times H\setminus\{0_{G'\times H}\}$. Indeed, $0_{G'\times H}\in S$ would imply that $\pi_H(x_{M_{\sigma(1)},j})=0_H$ and $[z_j]=0_{G'}=[0]$ for some index $j$ and hence, for $\ell$ large enough, $\pi_{\ell}(z_j)=\pi_{G_{M_{\sigma(\ell)}}}(x_{M_{\sigma(\ell)},j})=0_{G_{M_{\sigma(\ell)}}}$. But then, since $\pi_H(x_{M_{\sigma(\ell)},j})=\pi_H(x_{M_{\sigma(1)},j})=0_H$, we would have that $x_{M_{\sigma(\ell)},j}=0_{G_{M_{\sigma(\ell)}}\times H}$ which is in contradiction with the assumption that $S_{M_{\sigma(\ell)}}$ does not contain the zero. In a similar way we also exclude the possibility that $\sum S= 0_{G'\times H}$. Moreover, since $\pi_H(S)=\pi_H(S_{M_{\sigma(1)}})$ and $S_{M_{\sigma(1)}}$ is of type $\bm{\lambda}$, also $S$ is of type $\bm{\lambda}$.

Now we consider a permutation $\omega$ of $\Sym(k)$ and we define the ordering $$\bm{x}_{\omega}=(([z_{\omega(1)}],\pi_H(x_{M_{\sigma(1)},\omega(1)})),([z_{\omega(2)}],\pi_H(x_{M_{\sigma(1)},\omega(2)})),\ldots,([z_{\omega(k)}],\pi_H(x_{M_{\sigma(1)},\omega(k)})))$$ of $S$ and the ordering
$\bm{x}_{\ell,\omega}=(x_{M_{\sigma(\ell)},\omega(1)},x_{M_{\sigma(\ell)},\omega(2)},\ldots,$ $x_{M_{\sigma(\ell)}, \omega(k)})$ of $S_{M_{\sigma(\ell)}}$ that
corresponds to $\omega$.
As $f_{M_{\sigma(1)}}=f_{M_{\sigma(\ell)}}$ for all $\ell\geq 1$, there exists a pair
$(s,t)$ such that $y_s(\bm{x}_{\ell,\omega})=y_t(\bm{x}_{\ell,\omega})$ for all $\ell$.
Since $(\pi_{G_{M_{\sigma(1)}}}(y_s(\bm{x}_{1,\omega})),\pi_{G_{M_{\sigma(2)}}}(y_s(\bm{x}_{2,\omega})),\ldots,$
$ \pi_{G_{M_{\sigma(\ell)}}}(y_s(\bm{x}_{\ell,\omega})),\dots)$ belongs to the equivalence class $\pi_{G'}(y_s(\bm{x}_{\omega}))$ and since $\pi_{H}(y_s(\bm{x}_{1,\omega}))$ $= \pi_{H}(y_t(\bm{x}_{1,\omega}))$, it easily
follows that $y_s(\bm{x}_{\omega})=y_t(\bm{x}_{\omega})$ also for the set $S$. It means that $S$ is a non-zero sum subset of type $\bm{\lambda}$ that is not sequenceable. Since this is in contradiction with
Theorem~\ref{torsionfree}, we have proved the statement.
\end{proof}
\begin{cor}
Let $H$ be a finite abelian group, $k$ be a positive integer, and suppose that, for infinitely many primes $p$, any non-zero sum subset of size $k$ of $\Z_p\times H\setminus\{0_{\Z_p\times H}\}$ is sequenceable. Then, there exists a positive integer $N(\bm{\lambda})$ such that any non-zero sum subset of $G\times H\setminus\{0_{G\times H}\}$ is sequenceable provided that $G$ is abelian and $\vartheta(G)>N(\bm{\lambda})$.
\end{cor}
\begin{rem}
One can easily adapt the arguments of Theorem \ref{thm:as} to set $S$ whose sum is zero obtaining the following statement:

Let $H$ be a finite abelian group and suppose that, for infinitely many primes $p$, any subset of $\Z_p\times H$ of a given type $\bm{\lambda} = (\lambda_0, \ldots \lambda_{t-1} )$ is sequenceable. 
Then, there exists a positive integer $N(\bm{\lambda})$ such that any non-zero sum subset of type $\bm{\lambda}$ of $G\times H\setminus\{0_{G\times H}\}$ is sequenceable provided that $G$ is abelian and $\vartheta(G)>N(\bm{\lambda})$.
\end{rem}
Therefore, as a consequence of Theorem \ref{th:main}, we are able to prove our first asymptotic result.
\begin{thm}
Let $G$ be an abelian group such that $\vartheta(G)>N(k)$. Then subsets~$S$ of size~$k$ of~$G\times\Z_t\setminus\{0_{G\times\Z_t}\}$ are sequenceable in the following cases:
\begin{enumerate}
\item  $k \leq 11$ and $t \leq 5$,
\item  $k=12$ and $t \leq 4$,
\item  $k= 13$ and $t \in \{2,3\}$, provided~$S$ contains at least one element not in the subgroup $G\times \{0\}$,
\item  $k=14$ and $t = 2$,   provided~$S$ contains at least one element not in the subgroup $G\times \{0\}$,
\item $k=15$ and $t=2$, provided~$S$ does not contain exactly~$0$,~$1$,~$2$ or~$15$ elements of the subgroup $G\times \{0\}$.
\end{enumerate}
\end{thm}

\subsection{An Explicit Upper Bound}

Now, in case $G=\mathbb{Z}_m$, we provide an explicit upper bound on the numbers $N(\bm{\lambda})$ and $N(k)$.
First we need to recall some elementary facts about the solutions of a linear system over a field and over a commutative ring.
\begin{thm}[Corollary of Cramer's Theorem]\label{Cramer}
Let
\begin{equation}\label{sys1}\begin{cases}
m_{1,1}x_1+\dots+ m_{1,k}x_k=b_1;\\
\dots\\
m_{l,1}x_1+\dots+ m_{l,k}x_k=b_l
\end{cases} \end{equation}
be a linear system over a commutative ring $R$ whose associated matrix $M=(m_{i,j})$ is an $l\times k$ matrix over $R$. Let $(x_1,\dots, x_k)$ be a solution of the system and let $M'$ be a $k'\times k'$ square, nonsingular (i.e.~$\det(M')$ is invertible in $R$), submatrix of $M$. We assume, without loss of generality, that $M'$ is obtained by considering the first $k'$ rows and the first $k'$ columns of $M$. Then $x_1,\dots, x_k$ satisfy the following relation

$$\begin{bmatrix}
x_{1} \\
x_{2} \\
\vdots \\
x_{k'}
\end{bmatrix} = \begin{bmatrix}
m_{1,1} & m_{1,2} & m_{1,3} & \dots & m_{1,k'} \\
m_{2,1} & m_{2,2} & m_{2,3} & \dots & m_{2,k'} \\
\vdots & \vdots & \vdots & & \vdots \\
m_{k',1} & m_{k',2} & m_{k',3} & \dots & m_{k',k'}
\end{bmatrix} ^{-1} \begin{bmatrix}
b_1- (m_{1,k'+1}x_{k'+1}+\dots + m_{1,k}x_{k})\\
b_2- (m_{2,k'+1}x_{k'+1}+\dots+ + m_{2,k}x_{k})\\
\vdots \\
b_{k'}- (m_{k',k'+1}x_{k'+1}+\dots+ m_{k',k}x_{k})
\end{bmatrix} .$$

Equivalently, set $$ M'_j:=\begin{bmatrix}
m_{1,1} & m_{1,2} &\dots & m_{1,j-1} & b_1- (m_{1,k'+1}x_{k'+1}+\dots + m_{1,k}x_{k}) & \dots & m_{1,k'} \\
m_{2,1} & m_{2,2} &\dots & m_{2,j-1} & b_2- (m_{2,k'+1}x_{k'+1}+\dots+ + m_{2,k}x_{k}) & \dots & m_{2,k'} \\
\vdots & \vdots & \vdots & \vdots & \vdots & & \vdots \\
m_{k',1} & m_{k',2} &\dots & m_{k',j-1}& b_{k'}- (m_{k',k'+1}x_{k'+1}+\dots+ m_{k',k}x_{k})
& \dots & m_{k',k'}
\end{bmatrix}$$
we have that, for any $j\in [1,k']$
\begin{equation}\label{Cramer2}x_j= \frac{\det(M'_j)}{\det(M')}.\end{equation}
Let now assume that $R$ is a field and that the system \eqref{sys1} admits at least one solution. Then, given a maximal square nonsingular submatrix $M'$ of $M$, and fixed $x_{k'+1},\dots, x_k$ in $R$, we have that $x_1, x_2, \dots, x_k$ are a solution of the system \eqref{sys1} if and only if $x_1,\dots, x_{k'}$ satisfy the relation \eqref{Cramer2}.
\end{thm}
\begin{rem}\label{linarcombi}
In case the system \eqref{sys1} of Theorem \ref{Cramer} is homogeneous (i.e. $b_1=b_2=\dots=b_l=0$), due to the Laplace expansion of the determinant, $\det(M'_j)$ can be written as a linear combination of $x_{k'+1}, \dots, x_{k}$. More precisely, denoting by $M'_{i,j}$ the matrix obtained by deleting the $i$-th row and the $j$-th column of $M'$, we have that
\begin{equation}\label{linearcombieq}\det(M'_j)=\sum_{i=1}^{k'} (-1)^{j+i+1} \\det(M'_{i,j})(m_{i,k'+1}x_{k'+1}+\dots + m_{i,k}x_{k}).\end{equation}
\end{rem}
Now we are ready to provide, in case of groups of type $\mathbb{Z}_m\times H$, a quantitative version of Theorem \ref{thm:as}.
\begin{thm}\label{thm:explicitUB}
Let $H$ be a finite abelian group and suppose that, for infinitely many primes $p$, any non-zero sum subset of a given type $\bm{\lambda} = (\lambda_0, \ldots \lambda_{t-1} )$ of $\Z_p\times H\setminus \{0_{\Z_p\times H}\}$ is sequenceable. Then also any non-zero sum subset of type $\bm{\lambda}$ of $\mathbb{Z}_m\times H\setminus\{0_{\mathbb{Z}_m\times H}\}$ is also sequenceable provided that the prime factors of $m$ are all greater than $k!/2$ where $k=\lambda_0+\lambda_1+\dots+\lambda_{t-1}$.
\end{thm}
\begin{proof}
Let us assume, by contradiction, that there exists a non-zero sum subset $S=\{x_1,\dots,x_k\}$ of type $\bm{\lambda}$ in $\mathbb{Z}_m\times H\setminus\{0_{\mathbb{Z}_m\times H}\}$ that is not sequenceable and all the prime factors of $m$ are bigger than $k!/2$. Here $S$ would be a non-zero sum set of type $\bm{\lambda}$ whose elements do not admit an ordering with different partial sums.

This implies that, for any possible permutation $\omega\in \Sym(k)$, considered the ordering $$\bm{x}_{\omega}=(x_{\omega(1)},x_{\omega(2)},\dots,x_{\omega(k)}),$$
there exist $s$ and $t$ such that $y_{s}(\bm{x}_\omega)=y_{t}(\bm{x}_\omega)$, that is, assuming $s<t$, \begin{equation}\label{system}x_{\omega(s+1)}+x_{\omega(s+2)}+\dots+x_{\omega(t)}=0.\end{equation}
Therefore we would have a solution $(\pi_{\mathbb{Z}_m}(x_1),\dots,\pi_{\mathbb{Z}_m}(x_k))^T$ to the system of equations \eqref{system} in $\mathbb{Z}_m$. In the following, for simplicity, we will denote $\pi_{\mathbb{Z}_m}(x_j)$ by $x_j^1$.
Denoted by $M$ the $(k!)\times k$ matrix of the coefficients of this system, we note that the coefficients of $M$ belong to $\{0,1\}$. Hence, when it is clear from the context, we will not specify the ring $R$ over which we consider $M$ and its submatrices. Otherwise, we will use the notation $(M)_R$.

Since $\mathbb{Z}_m$ is a commutative ring, once we have chosen a maximal square nonsingular (over $\mathbb{Z}_m$) $k'\times k'$ submatrix $M'$ of $M$, because of Theorem \ref{Cramer}, any solution of the system \eqref{system} satisfies the relations \eqref{Cramer2}. We can also assume, without loss of generality, that $M'$ is the $k'\times k'$ matrix obtained by considering the first $k'$ rows and the first $k'$ columns of $M$. In view of Remark \ref{linarcombi}, for $j\leq k'$, we have that
$$x_j^1=\frac{\sum_{i=1}^{k'} (-1)^{j+i+1} \det(M'_{i,j})(m_{i,k'+1}x_{k'+1}^1+\dots + m_{i,k}x_{k}^1)}{\det(M')}$$
where we can perform the division since $M'$ is nonsingular in $\mathbb{Z}_m$ and, as defined in Remark \ref{linarcombi}, $M'_{i,j}$ is the matrix obtained from $M'$ by deleting the $i$-th row and the $j$-th column.

Here we note that since all the coefficients of $M$ belong to $\{0,1\}$, the determinant of any square submatrix of $M$ (that is big at most $k\times k$) is a sum of at most $k!/2$ ones and at most $k!/2$ minus ones. Therefore, if all the prime factors of $m$ are bigger than $k!/2$, a square submatrix $M'$ of $M$ is nonsingular in $\mathbb{Z}_{m}$ if and only if it is nonsingular in $\mathbb{Q}$.
It follows that, under this assumption, $M'$ is a maximal nonsingular square submatrix of $M$ also in $\mathbb{Q}$.
Now we consider the system \eqref{system} over $\mathbb{Q}$. Given an element $x\in \mathbb{Z}_m$, we denote by $\widetilde{x}$ the smallest non-negative integer that is equivalent to $x$ modulo $m$. Then we set
$$z_{j}^1=\begin{cases}
\frac{\sum_{i=1}^{k'} (-1)^{j+i+1} \det((M'_{i,j})_{\mathbb{Q}})(m_{i,k'+1}\widetilde{x^1_{k'+1}}+\dots + m_{i,k}\widetilde{x^1_{k}})}{\det((M')_{\mathbb{Q}})}\mbox{ if } j\leq k';\\
\widetilde{x_j^1} \mbox{ otherwise.}
\end{cases}
$$
We note that the system \eqref{system} is homogeneous and admits at least the trivial solution. Thus, because of Theorem \ref{Cramer}, $(z^1_1,\dots,z^1_k)^T$ is also a solution of \eqref{system} over $\mathbb{Q}$.
Now we define $z_j\in \mathbb{Q}\times H$ so that, $\pi_{\mathbb{Q}}(z_j)=z^1_j$ and $\pi_H(z_j)=\pi_H(x_j)$. Hence, because of the definition, $(z_1,\dots,z_k)^T$ is a solution to the linear system \eqref{system} in $\mathbb{Q}\times H$. Therefore, for any permutation $\omega$ of
$\widetilde{S}=\{z_1,\dots,z_k\}$, considered the ordering $$\bm{z}_{\omega}=(z_{\omega(1)},z_{\omega(2)},\dots,z_{\omega(k)}),$$ there exist $s,t \in [1,k]$ such that $y_s(\bm{z}_\omega)=y_t(\bm{z}_\omega)$.

Now we prove that the elements of $\widetilde{S}$ are different, non-zero and their sum is non-zero in $\mathbb{Q}\times H$.
To prove that $z_1,\dots,z_k$ are different, let us assume, by contradiction, that there exist $h$ and $j$ (with $h<j$) such that $z_h=z_j$ in $\mathbb{Q}\times H$.
Note that $\{x_1,\dots,x_k\}$ is a set of size $k$ in $\mathbb{Z}_{m}\times H$ and hence, either $\pi_H(x_h)\not=\pi_H(x_j)$ or $x^1_h\not=x^1_j$ in $\mathbb{Z}_{m}$.
In the first case, since $\pi_H(z_h)=\pi_H(x_h)$ and $\pi_H(z_j)=\pi_H(x_j)$ we would have that also $z_h\not=z_j$ which contradicts our hypothesis.
In the second one, we consider the projections $z^1_h$ and $z^1_j$ of $z_h$ and $z_j$ over the first component of $\mathbb{Q}\times H$.
Here, to obtain a contradiction, we consider three cases according to whether $h$ and $j$ belong to $[1,k']$ or not.

Case 1. Both $h$ and $j$ are smaller than or equal to $k'$. Because of the definition, we have that $$z^1_h=\frac{\sum_{i=1}^{k'} (-1)^{h+i+1} \det((M'_{i,h})_{\mathbb{Q}})(m_{i,k'+1}\widetilde{x^1_{k'+1}}+\dots + m_{i,k}\widetilde{x^1_{k}})}{\det((M')_{\mathbb{Q}})}=$$
$$=z^1_j=\frac{\sum_{i=1}^{k'} (-1)^{j+i+1} \det((M'_{i,j})_{\mathbb{Q}})(m_{i,k'+1}\widetilde{x^1_{k'+1}}+\dots + m_{i,k}\widetilde{x^1_{k}})}{\det((M')_{\mathbb{Q}})}.$$ This implies that
$$\sum_{i=1}^{k'} (-1)^{h+i+1} \det((M'_{i,h})_{\mathbb{Q}})(m_{i,k'+1}\widetilde{x^1_{k'+1}}+\dots + m_{i,k}\widetilde{x^1_{k}})=$$
$$=\sum_{i=1}^{k'} (-1)^{j+i+1} \det((M'_{i,j})_{\mathbb{Q}})(m_{i,k'+1}\widetilde{x^1_{k'+1}}+\dots + m_{i,k}\widetilde{x^1_{k}})$$ and hence that $$\sum_{i=1}^{k'} (-1)^{h+i+1} \det((M'_{i,h})_{\mathbb{Z}_m})(m_{i,k'+1}x^1_{k'+1}+\dots + m_{i,k}x^1_{k})=$$
$$=\sum_{i=1}^{k'} (-1)^{j+i+1} \det((M'_{i,j})_{\mathbb{Z}_m})(m_{i,k'+1}x^1_{k'+1}+\dots + m_{i,k}x^1_{k}).$$ But, due to property \eqref{Cramer2} and since $\det((M')_{\mathbb{Z}_{m}})$ is invertible in $\mathbb{Z}_{m}$, it would follows that also $x^1_h=x^1_j$ that is a contradiction.

Case 2. Here we assume that $h$ is smaller than or equal to $k'$ while $j$ is bigger than $k'$. Because of the definition, we have that $$z^1_h=\frac{
\sum_{i=1}^{k'} (-1)^{h+i+1} \det((M'_{i,h})_{\mathbb{Q}})(m_{i,k'+1}\widetilde{x^1_{k'+1}}+\dots + m_{i,k}\widetilde{x^1_{k}})}{\det((M')_{\mathbb{Q}})}=z^1_j=\widetilde{x^1_j}.$$ This implies that
$$\sum_{i=1}^{k'} (-1)^{h+i+1} \det((M'_{i,h})_{\mathbb{Z}_m})(m_{i,k'+1}x^1_{k'+1}+\dots + m_{i,k}x^1_{k})=\det((M')_{\mathbb{Z}_m})x^1_j.$$ But, due to property \eqref{Cramer2} and since $\det((M')_{\mathbb{Z}_m})$ is invertible in $\mathbb{Z}_m$, it would follows that also $$x^1_h=\frac{\sum_{i=1}^{k'} (-1)^{h+i+1} \det((M'_{i,h})_{\mathbb{Z}_m})(m_{i,k'+1}x^1_{k'+1}+\dots + m_{i,k}x^1_{k})}{\det((M')_{\mathbb{Z}_m})}=x^1_j$$
that is a contradiction.

Case 3. Both $h$ and $j$ are bigger than $k'$. This implies that $z^1_h=\widetilde{x^1_h}=\widetilde{x^1_j}=z^1_j$ and hence we would obtain that $x^1_h=x^1_j$ that is a contradiction.

Since we have obtained a contradiction in all cases, it follows that $z_1,\dots, z_k$ are all different.

Now we prove that none of the $z_j$ is $0_{\mathbb{Q}\times H}$. Indeed, $z_j=0_{\mathbb{Q}\times H}$ would imply that $\pi_H(z_j)=0_H$ and
$$\sum_{i=1}^{k'} (-1)^{j+i+1} \det((M'_{i,j})_{\mathbb{Q}})(m_{i,k'+1}\widetilde{x^1_{k'+1}}+\dots + m_{i,k}\widetilde{x^1_{k}})=0_{\mathbb{Q}}$$
or
$$\widetilde{x^1_j}=0_{\mathbb{Q}}$$
according to whether $j$ belongs to $[1,k']$ or not. Since $\det((M')_{\mathbb{Z}_{m}})$ is invertible in $\mathbb{Z}_m$, in both cases, it would follows that $x_j=0_{\mathbb{Z}_m\times H}$ that is a contradiction since $S$ does not contain the zero.
In a very similar way, one can check that $\sum \widetilde{S}\not= 0_{\mathbb{Q}\times H}$.

Finally, we note that, since the projections of $\{z_1,\dots z_k\}=\widetilde{S}$ over $H$ are equal to the projections of $\{x_1,\dots x_k\}=S$ over $H$ and $S$ is a set of type $\bm{\lambda}$, also $\widetilde{S}$ is of type $\bm{\lambda}$.

It follows that $\widetilde{S}$ is a non-zero sum subset of type $\bm{\lambda}$ in $\mathbb{Q}\times H\setminus\{0_{\mathbb{Q}\times H}\}$ that is not sequenceable. However, because of {\rm Theorem \ref{torsionfree}}, any non-zero sum subset $\widetilde{S}$ of type $\bm{\lambda}$ of  $ G\times H\setminus\{0_{G\times H}\}$ is sequenceable whenever $G$ is torsion-free and abelian. Since $\mathbb{Q}$ is a torsion-free abelian group, we obtain a contradiction. Therefore any non-zero sum subset of type $\bm{\lambda}$ of the group $\mathbb{Z}_m\times H$, that does not contain $0_{\mathbb{Z}_m\times H}$, is squenceable assuming that the prime factors of $m$ are all greater than $k!/2$.
\end{proof}
\begin{cor}
Let $H$ be a finite abelian group, $k$ be a positive integer, and suppose that, for infinitely many primes $p$, any non-zero sum subset of size $k$ of $\Z_p\times H\setminus\{0_{\Z_p\times H}\}$ is sequenceable. 
Then also any non-zero sum subset of size $k$ of $\mathbb{Z}_m\times H\setminus\{0_{\mathbb{Z}_m\times H}\}$ is sequenceable provided that the prime factors of $m$ are all greater than $k!/2$ where $k=\lambda_0+\lambda_1+\dots+\lambda_{t-1}$.
\end{cor}
\begin{rem}
One can easily adapt the arguments of Theorem \ref{thm:explicitUB} to set $S$ whose sum is zero obtaining the following statement:

Let $H$ be a finite abelian group and suppose that, for infinitely many primes $p$, any subset of a given type $\bm{\lambda} = (\lambda_0, \ldots \lambda_{t-1} )$ of $\Z_p\times H\setminus\{0_{\Z_p\times H}\}$ is sequenceable. Then also any subset of type $\bm{\lambda}$ of $\mathbb{Z}_m\times H\setminus\{0_{\mathbb{Z}_m\times H}\}$ is sequenceable provided that the prime factors of $m$ are all greater than $k!/2$ where $k=\lambda_0+\lambda_1+\dots+\lambda_{t-1}$.\end{rem}

Therefore, as a consequence of Theorem \ref{th:main}, we are able to prove Theorem \ref{th:infty}.
\\

\noindent\textbf{Theorem \ref{th:infty}.}
\emph{
Let $n = mt$ where all the prime factors of ~$m$ are bigger than $k!/2$. Then,  the subsets~$S$ of size~$k$ of~$\Z_n \setminus \{ 0 \}$ are sequenceable in the following cases:
\begin{enumerate}
\item  $k \leq 11$ and $t \leq 5$,
\item  $k=12$ and $t \leq 4$. 
\item  $k= 13$ and $t \in \{2,3\}$, provided~$S$ contains at least one element not in the subgroup of order~$m$,
\item  $k=14$ and $t = 2$,   provided~$S$ contains at least one element not in the subgroup of order~$m$,
\item $k=15$ and $t=2$, provided~$S$ does not contain exactly~$0$,~$1$,~$2$ or~$15$ elements of the subgroup of order~$m$.
\end{enumerate}
}

\end{document}